\documentclass[reqno, 11pt]{amsart}

\usepackage[utf8]{inputenc}
\usepackage[T1]{fontenc}

\usepackage{amsfonts,amssymb,enumerate,amsthm,amsmath,color,graphicx}
\usepackage{hyperref}
\usepackage[all]{xy}

\usepackage[a4paper,margin=4cm,includehead,includefoot,headsep=12pt,footskip=18pt]{geometry}

\newtheorem{theorem}{Theorem}[section]
\newtheorem{proposition}[theorem]{Proposition}

\newtheorem{remark}[theorem]{Remark}

\newtheorem{conjecture}{Conjecture}

\newcommand{\id}{\mathrm{id}}

\newcommand{\R}{\mathbb{R}}

\newcommand{\C}{\mathbb{C}}

\newcommand{\vol}{{\rm vol}}

\newcommand{\Aut}{\operatorname {Aut}}

\begin{document}

\title[On the Bergman metric of symmetric spaces]{On the Bergman metric of symmetric spaces}

\author{Andrea Loi}
\address{(Andrea Loi) Dipartimento di Matematica \\
         Universit\`a di Cagliari (Italy)}
         \email{loi@unica.it}

\author{Matteo Palmieri}
\address{(Matteo Palmieri) Dipartimento di Matematica \\
         Universit\`a di Cagliari (Italy)}
        \email{matteo.palmieri@unica.it}

\thanks{
The authors are supported by INdAM and  GNSAGA - Gruppo Nazionale per le Strutture Algebriche, Geometriche e le loro Applicazioni and by the project ProBiKi of Fondazione di Sardegna (Italy). The second author acknowledges the support and hospitality of Cergy University, where part of this work was conducted during a research stay.}

\subjclass[2000]{53C55, 32Q15, 53C24, 53C42 .} 
\keywords{Bergman kernel, Bergman metric, K{\"a}hler immersion, Symmetric space}

\begin{abstract}
We study bounded domains $\Omega\subset\mathbb{C}^n$ whose Bergman metric is locally symmetric, i.e. its Riemannian curvature tensor is parallel with respect to the Levi-Civita connection. Following the strategy developed in  \cite{UnifThm2}, we obtain two rigidity results. If the Bergman metric of $\Omega$ is complete, then $\Omega$ is (globally) symmetric. If instead $\Omega$ is pseudoconvex, then $\Omega$ is biholomorphic to $\widetilde\Omega\setminus E$, where $\widetilde\Omega\subset\mathbb{C}^n$ is a bounded symmetric domain and $E\subset\widetilde\Omega$ is relatively closed and pluripolar. The proofs combine the structure theory of Hermitian symmetric spaces with Calabi's theory of K\"ahler immersions into the infinite dimensional complex projective space (in particular, rigidity and the hereditary property of the diastasis), together with analytic and pluripotential tools based on extension properties of square-integrable holomorphic functions and the Bergman kernel.
\end{abstract}

\maketitle

\tableofcontents

\section{Introduction}
The Bergman kernel and the Bergman metric were introduced by S.~Bergman in 1922 (see \cite{Bergman1922}) and have since been intensively studied, as they provide rich analytic and geometric information on complex manifolds. In this paper, we focus on open and connected subsets of $\mathbb{C}^n$, $n \geq 1$, which we call domains. Recall that, for a bounded domain $\Omega \subset \mathbb{C}^n$, the associated Bergman space is the Hilbert space of holomorphic square-integrable functions
\begin{align*}
    A^2(\Omega) = \mathcal{O}(\Omega) \cap L^2(\Omega).
\end{align*}
If $\{ \phi_\alpha \}_{\alpha \in A}$ is a complete orthonormal system for $A^2(\Omega)$, the \emph{Bergman kernel} of $\Omega$ is given by the expansion
\begin{equation}\label{EQ-Kernelexpansion}
    K_\Omega(z, \xi) = \sum_{\alpha \in A}\ \phi_\alpha(z) \overline{\phi_\alpha(\xi)}.
\end{equation}
Remarkably, $K_\Omega(\cdot, \xi)$ satisfies the reproducing property
\begin{equation}\label{EQ-Reprodkernel}
    \varphi(\xi) = \langle \varphi, K_\Omega(\cdot, \xi) \rangle_{A^2(\Omega)}\qquad \forall\,\varphi\in A^2(\Omega).
\end{equation}
The Bergman kernel is closely related to the holomorphic and geometric structure of $\Omega$. It is well known that
\[
\Phi\in C^\infty(\Omega,\mathbb{R}),\quad \Phi(z):=\log K_\Omega(z,z),
\]
is strictly plurisubharmonic on $\Omega$. Hence
\begin{equation}\label{EQ-Bergmanform}
    \omega_\Omega = \sqrt{-1} \partial \bar\partial \Phi
\end{equation}
defines a K{\"a}hler form on $\Omega$, and the associated K{\"a}hler metric $g_\Omega$, which is a biholomorphic invariant, is called the \emph{Bergman metric} of $\Omega$. 

The basic example is the unit ball $\mathbb{B}^n$. If $\|\cdot\|$ denotes the Euclidean norm on $\mathbb{C}^n$, one computes
\begin{align*}
    K_{\mathbb{B}^n}(z, z) = \frac{n!}{\pi^n} \frac{1}{(1 - || z ||^2)^{n + 1}}\ ,\quad \omega_{\mathbb{B}^n} =- (n + 1) \sqrt{-1} \partial \bar\partial \log (1 - || z ||^2)
\end{align*}
In particular, $g_{\mathbb{B}^n}$ has constant holomorphic sectional curvature $-\frac{4}{n+1}$, as it differs only by the factor $n + 1$ from the hyperbolic metric on $\mathbb{B}^n$. The interested reader can refer to \cite{Krantz2013} for a detailed exposition on this topic. 

The studies accomplished so far highlight how $g_\Omega$ encodes the intrinsic geometric behavior of $\Omega$. One key example of such results has been given by Q.-K. Lu in 1966 (see \cite{UnifThm1}), where he establishes that if $g_\Omega$ is complete with constant holomorphic sectional curvature then $\Omega$ is biholomorphic to $\mathbb{B}^n$. In 2025, X. Huang, S.-Y. Li and J. N. Treuer (see \cite{UnifThm2}) proved that if $g_\Omega$ has constant holomorphic sectional curvature then it must be negative, and assuming pseudoconvexity instead of completeness, proved that $\Omega$ is biholomorphic to $\mathbb{B}^n$ with possibly a pluripolar set removed. Observe that this uniformization generalizes Lu's one, as in 1955 Bremermann proved (see \cite{Bremermann1955}) that if $g_\Omega$ is complete then $\Omega$ is pseudoconvex, while the converse is not true: the classical counterexample being the pseudoconvex domain obtained by removing a complex hyperplane from $\mathbb{B}^n$. Starting from \cite{UnifThm2}, in 2025 P. Ebenfelt, J. N. Treuer and M. Xiao (see \cite{UnifThm3}) maintained only the curvature hypothesis and obtained that $\Omega$ is biholomorphic to $\mathbb{B}^n$ with possibly a set of zero Lebesgue measure removed, over which $A^2$ functions extend holomorphically to $\mathbb{B}^n$.

The core matter of these results can be synthesized as follows: a local curvature condition on $g_\Omega$ determines, up to biholomorphisms, the global shape of $\Omega$. It is then fitting to ask whether the same concept holds for other kinds of curvature conditions and shapes. Within this perspective, it is well-known (see, e.g., \cite{KobayashiNomizu} Theorem 7.9) that the negative constant holomorphic curvature condition characterizes $\mathbb{B}^n$ in the class of simply-connected and complete K{\"a}hler manifolds, and remarkably (see \cite{Helgason} Ch. IV Theorem 5.6), the parallel Riemannian curvature condition characterizes Hermitian symmetric spaces in the same class. This suggests investigating whether the analogous local symmetry condition $\nabla R^\Omega=0$ for the Bergman metric forces $\Omega$ to be (essentially) a bounded symmetric domain. Alongside these uniformization-type results, recent works investigate finer Bergman-geometric features in settings closely related to bounded symmetric domains, such as metrics induced by the ball and the Bergman geometry of Cartan--Hartogs domains; see \cite{Palmieri2025BergmanBall, LoiMossaZuddas2025CartanHartogs} and references therein.

Inspired by such parallelism, our aim in this work is to generalize the above mentioned results to the setting of bounded symmetric domains in $\mathbb{C}^n$. Let $\text{Aut}(\Omega)$ denote the group of biholomorphisms of $\Omega$ onto itself. Recall that a domain $\Omega$ is called \emph{symmetric} if for every point $p\in\Omega$ there exists an automorphism
$f_p\in\text{Aut}(\Omega)$ such that $f_p^2=\id$ and $p$ is an isolated fixed point of $f_p$.
We say that $\Omega$ is \emph{locally symmetric} if for every point $p\in\Omega$ there exist a neighborhood $U$ of $p$
and an automorphism $f_p\in\text{Aut}(U)$ such that $f_p^2=\id$ and $p$ is an isolated fixed point of $f_p$.
Equivalently (see \cite{Helgason}, Ch.~IV, Theorem~2.1), if $\nabla$ and $R^\Omega$ denote the Levi--Civita connection
and the Riemannian curvature tensor of $g_\Omega$, then $\Omega$ is locally symmetric if and only if $\nabla R^\Omega=0$.
We refer to \cite{Helgason} for further background on symmetric spaces.

In analogy with \cite{UnifThm1} and \cite{UnifThm2}, the main results of the present paper are as follows:

\begin{theorem}\label{THM-completelocsymmetric}
Let $\Omega \subset \mathbb{C}^n$, $n \geq 1$, be a bounded domain. If $\nabla R^\Omega = 0$ and $g_\Omega$ is complete, then $\Omega$ is symmetric.
\end{theorem}

\begin{theorem}\label{THM-pseudoconvexlocsymmetric}
Let $\Omega \subset \mathbb{C}^n$, $n \geq 1$, be a bounded domain. If $\nabla R^\Omega = 0$ and $\Omega$ is pseudoconvex, then there exists $\widetilde\Omega \subset \mathbb{C}^n$ bounded symmetric domain, and  $E \subset \widetilde\Omega$ pluripolar and closed in $\widetilde\Omega$ such that $\Omega \cong \widetilde\Omega \setminus E$. 
\end{theorem}

We follow the general strategy developed in \cite{UnifThm2}. 
The main tools are Calabi's theory of K\"ahler immersions into $\mathbb{CP}^\infty$, the structure theory of Hermitian symmetric spaces, and pluripotential methods based on analytic properties of the Bergman kernel.

\medskip
\noindent\textbf{Organization of the paper.}
Section~2 is devoted to the proofs of Theorems~\ref{THM-completelocsymmetric} and~\ref{THM-pseudoconvexlocsymmetric}.
In the final part we include several remarks and we formulate a conjecture, inspired by the uniformization results in \cite{UnifThm3}, about removing the pseudoconvexity assumption.

\section{Proofs of the main results}\label{SUB-Preliminaries}

In the proofs of Theorem \ref{THM-completelocsymmetric} and \ref{THM-pseudoconvexlocsymmetric}, we need the following technical results about the Bergman metric of bounded domains.

We start with a general fact, valid for any bounded domain $\Omega\subset\C^n$, concerning the injectivity of K\"ahler immersions of positive scalings of the Bergman metric into $\mathbb{CP}^\infty$.

\begin{proposition}\label{PROP-injectiveKahlerimmersion}
Let $\Omega \subset \mathbb{C}^n$, $n \geq 1$, be a bounded domain. If for $c \in \mathbb{R}^{> 0}$, $(\Omega, c g_\Omega)$ admits a K{\"a}hler immersion into $\mathbb{CP}^\infty$, then such immersion is injective.
\end{proposition}
\begin{proof}
Let $\mathcal{S} = \{\phi_j\}_{j\in\mathbb{N}^*}$ be a complete orthonormal system of $A^2(\Omega)$ and consider the associated Bergman--Bochner map
\begin{equation}\label{EQ-BergmanBochnermap}
\mathcal{B}^{\mathcal{S}} \colon(\Omega,g_\Omega) \longrightarrow \mathbb{CP}^\infty,\qquad z\longmapsto [\phi_1(z): \phi_2(z): \ldots ].
\end{equation}
It is well-known (see, e.g., \cite[Sec.~3.4]{LoiZeddabook}) that $\mathcal{B}^{\mathcal{S}}$ is a full K{\"a}hler immersion.

\smallskip
\noindent\emph{Step 1: $\mathcal{B}^{\mathcal{S}}$ is injective.}
Assume by contradiction $\mathcal{B}^{\mathcal{S}}(p) = \mathcal{B}^{\mathcal{S}}(q)$ for some $p, q \in \Omega$, $p \neq q$, so there exists $\lambda\in\mathbb{C}^*$ such that $\phi_j(p) = \lambda\, \phi_j(q)$ for all $j\in\mathbb{N}^*$. Let $K_\Omega$ be the Bergman kernel of $\Omega$. Using the expansion \eqref{EQ-Kernelexpansion}
\[
K_\Omega(\cdot,p) = \sum_{j\ge 1}\phi_j(\cdot)\,\overline{\phi_j(p)}
\]
we obtain
\[
K_\Omega(\cdot,p) = \overline{\lambda}\,K_\Omega(\cdot,q).
\]
By the reproducing property \eqref{EQ-Reprodkernel}, for every $\varphi \in A^2(\Omega)$,
\[
\varphi(p) = \langle \varphi, K_\Omega(\cdot,p) \rangle_{A^2(\Omega)}
= \lambda\, \langle \varphi, K_\Omega(\cdot,q) \rangle_{A^2(\Omega)}
= \lambda\,\varphi(q),
\]
hence
\begin{equation}\label{EQ-mapsatdifferentpoints}
\qquad \varphi(p) = \lambda\,\varphi(q),\quad \forall\, \varphi \in A^2(\Omega).
\end{equation}
Since $\Omega$ is bounded, every polynomial belongs to $A^2(\Omega)$; in particular, if $p = (p_1, \dots, p_n)$ and $q = (q_1, \dots, q_n)$, there exists $k \in \{1, \dots, n\}$ with $p_k \neq q_k$, so the holomorphic function $\varphi(z) = z_k - p_k$ lies in $A^2(\Omega)$ and satisfies $\varphi(p)=0$ while $\varphi(q) = q_k - p_k \neq 0$, contradicting \eqref{EQ-mapsatdifferentpoints}. Therefore $\mathcal{B}^{\mathcal{S}}$ is injective.

\smallskip
\noindent\emph{Step 2: any K{\"a}hler immersion $(\Omega,cg_\Omega) \to \mathbb{CP}^\infty$ is injective.}
Let $F \colon (\Omega,cg_\Omega) \to \mathbb{CP}^\infty$ be a K{\"a}hler immersion. Recall that the diastasis of the Fubini--Study metric $g_{FS}$ can be written in homogeneous coordinates $[Z], [W] \in \mathbb{CP}^\infty$ as
\[
D_{FS}([Z],[W]) = \log\frac{\langle Z, Z\rangle\, \langle W, W\rangle}{|\langle Z, W\rangle|^2}\ge 0,
\]
and $D_{FS}([Z],[W]) = 0$ if and only if $[Z] = [W]$ (by equality in Cauchy--Schwarz). Moreover, by Calabi's hereditary property of the diastasis (see, e.g., \cite{Calabi53}), the diastasis of $c \Omega := (\Omega,cg_\Omega)$ is given by the pullback of $D_{FS}$ via $F$, i.e.
\[
D^{c\Omega}(x, y) = D_{FS} \bigl(F(x),F(y) \bigr),
\]
whenever the diastasis is defined. Assume $F(p) = F(q)$. Then 
$$D_{FS}(F(p),F(q))=D^{c\Omega}(p,q) = 0.$$ 
Since the diastasis scales linearly with the metric, i.e. $D^{c\Omega} = c\,D^\Omega$, we also have 
$D^\Omega(p, q) = 0$. Again by Calabi's hereditary property of the diastasis (see, e.g., \cite{Calabi53}), if $\mathcal{B}^\mathcal{S}$ is a Bergman--Bochner map as in \eqref{EQ-BergmanBochnermap}, then
\begin{align*}
    D_{FS} \bigl(\mathcal{B}^\mathcal{S}(p), \mathcal{B}^\mathcal{S}(q) \bigr) = 0,
\end{align*}
which forces $\mathcal{B}^\mathcal{S}(p) = \mathcal{B}^\mathcal{S}(q)$, and $p = q$. Hence $F$ is injective.
\end{proof}

In the next proposition we collect three properties that will be repeatedly used in the sequel. The striking point in (A) is that, contrary to the ball $\mathbb{B}^n$ for which every positive multiple of the Bergman metric is projectively induced, for a general bounded symmetric domain projective inducedness occurs only for scalings prescribed by the Wallach set. On the other hand, (B) and (C) reflect features that are shared with $\mathbb{B}^n$.

\begin{proposition}\label{PROP-symmetricdomains}
Let $\Omega \subset \mathbb{C}^n$, $n \ge 1$, be a bounded symmetric domain. Then:
\begin{itemize}
\item [(A)] there exists a constant $\Gamma_\Omega \in \mathbb{R}_{\ge 0}$ such that, for every $\delta>\Gamma_\Omega$,
 the scaled Bergman manifold $(\Omega,\delta g_\Omega)$ admits a full and injective K\"ahler immersion into $\mathbb{CP}^\infty$;
\item [(B)] the real analytic expansion at $0$ of the diagonal Bergman kernel $z \mapsto K_\Omega(z, z)$ contains no non-constant purely holomorphic terms and no non-constant purely antiholomorphic terms.
\item [(C)] For any $z \in \Omega$ there is a constant $C_z \in \mathbb{R}^{>0}$ dependent on $z$ such that $| K_\Omega(z, w) | \leq C_z,\, \forall\, w \in \Omega$.

\end{itemize}
\end{proposition}
\begin{proof}
Let $\Omega\simeq \Omega_1 \times \cdots \times \Omega_s$ be the decomposition of the bounded symmetric domain into irreducible factors. For every $k \in \{ 1, \dots, s \}$, denote by
\begin{itemize}
    \item $W_c(\Omega_k)$ the continuous part of the Wallach set of $\Omega_k$;
    \item $\gamma_k$ the genus of $\Omega_k$;
    \item $r_k$ the rank of $\Omega_k$.
\end{itemize}

\smallskip
\noindent\emph{(A)}
Set
\begin{equation}\label{EQ-GammaOmega}
\Gamma_\Omega := \max_{1 \le k \le s}\left\{\frac{\inf W_c(\Omega_k)}{\gamma_k}\right\} \in \mathbb{R}_{\ge 0}.
\end{equation}
If $\delta > \Gamma_\Omega$, then $\delta > \inf W_c(\Omega_k)/ \gamma_k$ for every $k$, hence $\delta\,\gamma_k\in W_c(\Omega_k)$ for all $k$. By \cite[Theorem~2]{LoiZedda} it follows that, for every
$k = 1, \dots, s$, the scaled Bergman manifold $(\Omega_k,\delta g_{\Omega_k})$ admits a \emph{full} K\"ahler immersion
\[
F_k \colon (\Omega_k,\delta g_{\Omega_k}) \longrightarrow \mathbb{CP}^\infty.
\]
Consider the product immersion $F_1 \times \cdots \times F_s$ and compose it with the (infinite-dimensional) Segre embedding to obtain a full K\"ahler immersion
\[
F \colon (\Omega,\delta g_\Omega) \cong (\Omega_1, \delta g_{\Omega_1}) \times \cdots \times(\Omega_s, \delta g_{\Omega_s})
\longrightarrow \mathbb{CP}^\infty.
\]
Since $\Omega$ is bounded, Proposition~\ref{PROP-injectiveKahlerimmersion} implies that \emph{any} such K\"ahler immersion is injective. This proves (A).

\smallskip
\noindent\emph{(B)}
Fix $k \in \{1,\dots,s\}$ and let $N_{\Omega_k}$ be the generic norm of $\Omega_k$. It is well-known (see, e.g., \cite[Sec.~1.3]{RoosYin}) that on the diagonal one has
\begin{equation}\label{EQ-Kernelandnorm-better}
K_{\Omega_k}(z, z) = \frac{1}{\mathrm{vol}(\Omega_k)}\bigl(N_{\Omega_k}(z,\bar z)\bigr)^{-\gamma_k},
\end{equation}
and there are $m_1, \dots, m_{r_k}$ polynomials of of bidegrees $(1, 1),\dots, (r_k, r_k)$ on $\Omega_k \times \text{conj}(\Omega_k)$ respectively, such that
\begin{equation}\label{EQ-genericnorm-better}
N_{\Omega_k}(z,\bar z) = 1+\sum_{\ell = 1}^{r_k} (-1)^\ell\, m_\ell(z,\bar z).
\end{equation}
In particular, each $m_\ell(z,\bar z)$ is a sum of monomials involving \emph{both} $z$ and $\bar z$, so the real analytic germ
\[
H_k(z,\bar z) := \sum_{\ell = 1}^{r_k} (-1)^\ell\, m_\ell(z,\bar z)
\]
has no non-constant purely holomorphic terms (depending only on $z$) and no non-constant purely antiholomorphic terms (depending only on $\bar z$). Now observe that if $H(z,\bar z)$ has this property, then so does any power $H(z,\bar z)^q$ ($q\ge1$), since products of mixed monomials are still mixed. Therefore, for any real $\alpha$, the binomial expansion
\[
(1 + H(z,\bar z))^{\alpha} = \sum_{q\ge0}\binom{\alpha}{q}\,H(z,\bar z)^q
\]
cannot contain non-constant purely holomorphic or antiholomorphic terms.
Applying this with $H = H_k,\, \alpha = -\gamma_k$, and using \eqref{EQ-Kernelandnorm-better}--\eqref{EQ-genericnorm-better} we conclude that the real analytic expansion of $z\mapsto K_{\Omega_k}(z, z)$ at $0$
contains no non-constant purely holomorphic nor purely antiholomorphic terms. Finally, recall that the Bergman kernel factorizes on products, i.e. writing $z = (z^{(1)},\dots,z^{(s)}) \in \Omega_1 \times \cdots \times \Omega_s$
\begin{equation}\label{EQ-Bergmanfactors}
    K_{\Omega}(z, w) = \prod_{k = 1}^s K_{\Omega_k} \bigl( z^{(k)}, w^{(k)} \bigr).
\end{equation}
Hence the same conclusion holds for $K_{\Omega}(z, z)$. This proves (B).

\smallskip
\noindent\emph{(C)}
Assume $\Omega$ is irreducible. Fix $z \in \Omega$. Since $\Omega$ is symmetric, $\Aut(\Omega)$ acts transitively; hence there exists
$\varphi^z \in \Aut(\Omega)$ such that $\varphi^z(z) = 0$. Let $J_{\varphi^z}(\xi) := \det_{\C}(d(\varphi^z)_\xi)$ be the complex Jacobian determinant of $\varphi^z$. In its Harish--Chandra realization, $\Omega$ is a bounded (positive) circular domain with $0 \in \Omega$, hence by \cite[Theorem~1]{Kaup1970}
every $\varphi \in \Aut(\Omega)$ extends holomorphically to $\overline{\Omega}$; in particular, $\varphi^z$ (and thus $J_{\varphi^z}$) extends holomorphically to $\overline{\Omega}$. Therefore, being $\overline{\Omega}$ compact and $\varphi^z$ biholomorphic
\begin{align*}
    M_z := \sup_{\xi \in \Omega}|J_{\varphi^z}(\xi)| \in \mathbb{R}^{> 0}
\end{align*}
We now use the transformation rule for the Bergman kernel under biholomorphisms:
for every $\psi\in\Aut(\Omega)$ and all $u,v\in\Omega$,
\begin{equation}\label{EQ-KernelTransform}
K_\Omega(u,v) = J_\psi(u)\,\overline{J_\psi(v)}\,K_\Omega(\psi(u),\psi(v)).
\end{equation}
Applying \eqref{EQ-KernelTransform} with $\psi = \varphi^z$ and $u = z$, we obtain
\begin{equation}\label{EQ-ReduceToZero}
K_\Omega(z, w) = J_{\varphi^z}(z)\,\overline{J_{\varphi^z}(w)}\, K_\Omega(0,\varphi^z(w)), \qquad w \in \Omega.
\end{equation}
Next we claim that $K_\Omega(0,\cdot)$ is constant. Indeed, since in its Harish--Chandra realization $\Omega$ is circled, for every $\theta \in \R$ the rotation $r_\theta(w) = e^{\sqrt{-1}\theta}w$ belongs to $\Aut(\Omega)$ and satisfies $J_{r_\theta}\equiv e^{\sqrt{-1}n\theta}$.
Using \eqref{EQ-KernelTransform} with $\psi = r_\theta$ and $u = 0$ gives
\[
K_\Omega(0,w) = K_\Omega(0, e^{\sqrt{-1}\theta}w),\qquad \forall\, \theta \in \R,\ \forall\,w\in\Omega.
\]
Since $w \mapsto K_\Omega(0,w)$ is antiholomorphic, its power series at $0$ has the form $K_\Omega(0, w) = \sum_\alpha c_\alpha\, \overline{w^\alpha}$; the above invariance forces $c_\alpha = 0$ for all $|\alpha| \ge 1$. Thus, $K_\Omega(0, w)\equiv K_\Omega(0, 0) = \frac{1}{\vol(\Omega)}$. Plugging this into \eqref{EQ-ReduceToZero} yields
\[
|K_\Omega(z, w)| = \frac{|J_{\varphi^z}(z)|}{\vol(\Omega)}\, |J_{\varphi^z}(w)| \leq \frac{|J_{\varphi^z}(z)|}{\vol(\Omega)}\, M_z, 
\qquad\forall\, w \in \Omega.
\]
Setting $C_z: = \dfrac{|J_{\varphi^z}(z)|}{\vol(\Omega)}\,M_z \in \mathbb{R}^{> 0}$ proves (C) in the irreducible case, hence in general by \eqref{EQ-Bergmanfactors}.
\end{proof}

\subsection{Proof of Theorem \ref{THM-completelocsymmetric}}\label{SUB-Proof1}

\begin{proof}
Let $\pi\colon (X,\pi^{*}g_\Omega)\to (\Omega,g_\Omega)$ be the universal covering. Since $g_\Omega$ is complete and $\nabla R^\Omega = 0$, $(X,\pi^{*}g_\Omega)$ is simply connected, complete and locally symmetric, hence a (globally) Hermitian symmetric space of noncompact type. Therefore there exist a bounded symmetric domain $\widetilde\Omega \subset \mathbb{C}^n$ and a constant
$\lambda \in \mathbb{R}^{> 0}$ together with a biholomorphic isometry
\[
\Psi\colon (X,\pi^{*}g_\Omega)\longrightarrow (\widetilde\Omega,\lambda g_{\widetilde\Omega})
\]
(see \cite[Ch.~IV, Theorem~5.6 and Ch.~VIII, Theorem~7.1]{Helgason}).
Choose an integer $k > \Gamma_{\widetilde\Omega}/\lambda$, where $\Gamma_{\widetilde\Omega}$ is given by \eqref{EQ-GammaOmega}. By Proposition~\ref{PROP-symmetricdomains} (A), the scaled Bergman manifold $(\widetilde\Omega,k\lambda g_{\widetilde\Omega})$ admits a full and injective K\"ahler immersion
\[
\mathcal{B}^{\widetilde\Omega}\colon (\widetilde\Omega,k\lambda g_{\widetilde\Omega})\longrightarrow \mathbb{CP}^\infty .
\]
On the other hand, fix a complete orthonormal system $\mathcal{S}$ of $A^2(\Omega)$ and consider the Bergman--Bochner map $\mathcal{B}^{\mathcal{S}}\colon (\Omega,g_\Omega)\to\mathbb{CP}^\infty$ given by \eqref{EQ-BergmanBochnermap}. Composing with the (generalized) Veronese map $\mathcal{V}_k\colon\mathbb{CP}^\infty\to\mathbb{CP}^\infty$ (so that $\mathcal{V}_k^{*}\omega_{FS} = k\omega_{FS}$), we obtain a full K\"ahler immersion
\[
\mathcal{B}^{\Omega}:=\mathcal{V}_k\circ \mathcal{B}^{\mathcal{S}}\colon (\Omega,kg_\Omega)\longrightarrow \mathbb{CP}^\infty .
\]
Pulling back via the covering and using the identification given by $\Psi$, we get a full K\"ahler immersion
\[
\widehat{\mathcal{B}} := \mathcal{B}^{\Omega} \circ \pi \circ \Psi^{-1} \colon (\widetilde\Omega, k\lambda g_{\widetilde\Omega})
\longrightarrow \mathbb{CP}^\infty.
\]
By Calabi's rigidity theorem \cite[Theorem~9]{Calabi53}, there exists a rigid motion $T$ of $\mathbb{CP}^\infty$ such that $\widehat{\mathcal{B}} = T\circ \mathcal{B}^{\widetilde\Omega}$.
Since $\mathcal{B}^{\widetilde\Omega}$ is injective, also $\widehat{\mathcal{B}}$ is injective; hence
$\pi$ must be injective. Therefore the universal covering is trivial and $\Omega$ is biholomorphic to $\widetilde\Omega$. In particular, $\Omega$ is symmetric.
\end{proof}

\subsection{Proof of Theorem \ref{THM-pseudoconvexlocsymmetric}}\label{SUB-Proof2}

Fix $p \in \Omega$. By \cite[Ch.~IV, Theorem~5.1 and Ch.~VIII, Theorem~7.1]{Helgason} there exist  $\widetilde\Omega \subset \mathbb{C}^n$ bounded symmetric domain, $U \subset \Omega$ neighborhood of $p$, $V \subset \widetilde\Omega$ open, $\lambda \in \mathbb{R}^{> 0}$ and a biholomorphic isometry
\begin{align*}
    f \colon (U, g_\Omega |_U) \to (V, \lambda g_{\widetilde\Omega} |_V)
\end{align*}
If $\iota \colon V \to \widetilde\Omega$ is the inclusion, then $\iota \circ f \colon (U, g_\Omega |_U) \to (\widetilde\Omega, \lambda g_{\widetilde\Omega})$ is a full K{\"a}hler immersion. Hence, if we pick $k \in \mathbb{N}^*,\, k > \frac{\Gamma_{\widetilde\Omega}}{\lambda}$, where $\Gamma_{\widetilde\Omega}$ is given by \eqref{EQ-GammaOmega}, by Proposition \ref{PROP-symmetricdomains} (A) we have a full and injective K{\"a}hler immersion
\begin{align*}
    \mathcal{B}^{\widetilde\Omega} \colon (\widetilde\Omega, k\lambda g_{\widetilde\Omega}) \to \mathbb{CP}^\infty
\end{align*}
Then $\mathcal{B}^{\widetilde\Omega} \circ \iota \circ f \colon (U, kg_\Omega |_U) \to \mathbb{CP}^\infty$ is still a full K{\"a}hler immersion. Composing the (generalized) Veronese map $\mathcal{V}_k \colon \mathbb{CP}^\infty \to \mathbb{CP}^\infty$, which satisfies $\mathcal{V}_k^*(\omega_{FS}) = k \omega_{FS}$, with a Bergman-Bochner map $\mathcal{B}^\mathcal{S} \colon (\Omega, g_{\Omega}) \to \mathbb{CP}^\infty$ given by \eqref{EQ-BergmanBochnermap}, we have a full K{\"a}hler immersion
\begin{align*}
    \mathcal{B}^\Omega := \mathcal{V}_k \circ \mathcal{B}^\mathcal{S} \colon (\Omega, k g_\Omega) \to \mathbb{CP}^\infty
\end{align*}
so by Calabi's rigidity theorem \cite[Theorem~9]{Calabi53} there exists a rigid motion $T$ of $\mathbb{CP}^\infty$ such that
\begin{align*}
    \mathcal{B}^{\widetilde\Omega} \circ \iota \circ f = T \circ \mathcal{B}^\Omega |_U
\end{align*}
Let $q \in \Omega$ and  $\eta_q \colon [0, 1] \to \Omega$ be a simple continuous path with $\eta_q(0) = p,\ \eta_q(1) = q$. Repeatedly applying Calabi's rigidity theorem, we find $A_q$ connected neighborhood of $\eta_q([0, 1]) \cup U$ and a K{\"a}hler immersion
\begin{align*}
    \widetilde f_q \colon (A_q, kg_\Omega |_{A_q}) \to (\widetilde\Omega, k\lambda g_{\widetilde\Omega}) 
\end{align*}
such that $\widetilde f_q |_U = \iota \circ f$. Then $T \circ \mathcal{B}^\Omega |_U = \mathcal{B}^{\widetilde\Omega} \circ \widetilde f_q |_U$, which by the identity principle for holomorphic functions yields
\begin{align*}
    T \circ \mathcal{B}^\Omega |_{A_q} = \mathcal{B}^{\widetilde\Omega} \circ \widetilde f_q.
\end{align*}
Furthermore, if $r \in \Omega \setminus \{ q \}$ and $A_q \cap A_r \neq \emptyset$, we have
\begin{align*}
    \mathcal{B}^{\widetilde\Omega} \circ \widetilde f_q |_{A_q \cap A_r} = T \circ \mathcal{B}^\Omega |_{A_q \cap A_r} = \mathcal{B}^{\widetilde\Omega} \circ \widetilde f_r |_{A_q \cap A_r}
\end{align*}
but $\mathcal{B}^{\widetilde\Omega}$ is injective, so $\widetilde f_q |_{A_q \cap A_r} = \widetilde f_r |_{A_q \cap A_r}$. Consequently, there is a well defined K{\"a}hler immersion 
\begin{align*}
    \widetilde f \colon (\Omega, kg_{\Omega}) \to (\widetilde\Omega, k\lambda g_{\widetilde\Omega}), \  \widetilde f (s) := \widetilde f_s(s)
\end{align*}
and since by construction $T \circ \mathcal{B}^\Omega |_U = \mathcal{B}^{\widetilde\Omega} \circ \widetilde f |_U$, by the identity principle for holomorphic functions
\begin{align*}
    T \circ \mathcal{B}^\Omega = \mathcal{B}^{\widetilde\Omega} \circ \widetilde f.
\end{align*}
In particular, $\widetilde f$ is injective because $T \circ \mathcal{B}^\Omega$ is injective, by Proposition \ref{PROP-injectiveKahlerimmersion}. By the open mapping theorem and the pseudoconvexity of $\Omega$, $W := \widetilde f (\Omega)$ is a pseudoconvex domain in $\widetilde \Omega$ and
$\widetilde f \colon (\Omega, g_{\Omega}) \to (W, \lambda g_{\widetilde\Omega} |_W)$
is a biholomorphic isometry. By the biholomorphic invariance of the Bergman metric, we also get
\begin{align*}
    \lambda g_{\widetilde\Omega} |_W = \big( \widetilde f ^{-1} \big)^* g_\Omega = g_W
\end{align*}
which translates, taking K{\"a}hler potentials, to
\begin{equation}\label{EQ-differencepotentials1}
    \partial \bar\partial \left( \log(K_W(z, z)) - \lambda \log(K_{\widetilde\Omega}(z, z)) \right) = 0\ ,\quad z \in W.
\end{equation}
Due to the transitivity of the action of $\text{Aut}(\widetilde\Omega)$ onto $\widetilde\Omega$, we can assume $0 \in W$. Let $\{ \phi_j \}_{j \in \mathbb{N}^*}$ be a complete orthonormal system for $A^2(W)$ such that $\phi_1(0) > 0$ and $\phi_j(0) = 0$, for all $j \geq 2$. Let $\Lambda' \subseteq W \setminus \phi_1^{-1}(0)$ be a neighborhood of $0$. Since $\phi_1 \in \mathcal{O}(\Lambda')$, \eqref{EQ-differencepotentials1} results into
\begin{equation}\label{EQ-differencepotentials2}
    \partial \bar\partial \left( \log(K_W(z, z)) - \log(|\phi_1(z)|^2) - \lambda \log(K_{\widetilde\Omega}(z, z)) \right) = 0\ ,\quad z \in \Lambda'
\end{equation}
that is, $z \mapsto \log(K_W(z, z)) - \log(|\phi_1|^2(z)) - \lambda \log(K_{\widetilde\Omega}(z, z))$ is pluriharmonic on $\Lambda'$. Consequently there exist $\Lambda \subseteq \Lambda'$ neighborhood of $0$, $h \in \mathcal{O}(\Lambda),\ h = \sum_{|\alpha| \geq 0} a_\alpha z^\alpha$, such that
\begin{equation}\label{EQ-pluriharmonicity}
\begin{split}
    \log(K_W(z, z)) &- \log(|\phi_1(z)|^2) - \lambda \log(K_{\widetilde\Omega}(z, z)) = 2\Re(h(z)) =\\
    &= h(z) + \overline{h(z)} = \sum_{|\alpha| \geq 0} \big( a_\alpha z^\alpha + \overline{a_\alpha} \overline{z^\alpha} \big)\ ,\quad z \in \Lambda.
\end{split}
\end{equation}
We now show that  $a_\alpha = 0$, for all $\alpha \in \mathbb{N}^n,\ |\alpha| \geq 1$. First, by the properties of the Bergman kernel: $K_{\widetilde\Omega}(0, 0) = \frac{1}{\text{vol}(\widetilde\Omega)}$ so we can compute around $0$
\begin{align*}
    - \log(K_{\widetilde\Omega}(z, z)) = \log\big({\text{vol}(\widetilde\Omega)}\big)+ \sum_{k = 1}^\infty \frac{\Big(1 - \text{vol}(\widetilde\Omega) K_{\widetilde\Omega}(z, z)\Big)^k}{k}.
\end{align*}
We infer by Proposition \ref{PROP-symmetricdomains} (B) that the real analytic expansion of $z \mapsto - \lambda \log(K_{\widetilde\Omega}(z, z))$ at $0$ does not contain non-constant purely holomorphic or antiholomorphic terms. Secondly, by the properties of the Bergman kernel
\begin{align*}
    \log(K_W(z, z)) - \log(|\phi_1(z)|^2) = \log \left( 1 + \sum_{j = 2}^\infty \frac{|\phi_j(z)|^2}{|\phi_1(z)|^2} \right).
\end{align*}
Since for all $j \geq 2$,  $\frac{\phi_j}{\phi_1} \in \mathcal{O}(\Lambda), \frac{\phi_j}{\phi_1}(0) = 0$, around $0$ we have $\frac{\phi_j}{\phi_1} = \sum_{| \beta | \geq 1} b_{j, \beta}\, z^\beta$, thus
\begin{align*}
    \frac{|\phi_j(z)|^2}{|\phi_1(z)|^2} = \sum_{|\beta|, |\gamma| \geq 1} b_{j, \beta} \overline{b_{j, \gamma}}\ z^\beta \overline{z^\gamma}
\end{align*}
and again, since around $0$
\begin{align*}
    \log \left( 1 + \sum_{j = 2}^\infty \frac{|\phi_j(z)|^2}{|\phi_1(z)|^2} \right) = \sum_{k = 1}^{\infty} \frac{(-1)^{k + 1}}{k}\left( \sum_{j = 2}^\infty \frac{|\phi_j(z)|^2}{|\phi_1(z)|^2} \right)^k
\end{align*}
we conclude that the real analytic expansion of $z \mapsto \log(K_W(z, z)) - \log(|\phi_1|^2(z))$ at $0$ does not contain non-constant purely holomorphic or antiholomorphic terms. Comparing with \eqref{EQ-pluriharmonicity} gives the claim. Hence
\begin{align*}
    \log(K_W(z, z)) - \log(|\phi_1(z)|^2) - \lambda \log(K_{\widetilde\Omega}(z, z)) = 2\Re(a_0)\ ,\quad z \in \Lambda
\end{align*}
or equivalently, if we set $c_0 := e^{2\Re(a_0)} \in \mathbb{R}^{> 0}$
\begin{equation}\label{EQ-nopureterms}
    K_W(z, z) = c_0 |\phi_1(z)|^2 \big( K_{\widetilde\Omega}(z, z) \big)^\lambda\ ,\quad z \in \Lambda .
\end{equation}
After analytic continuation on a neighborhood of $(0, 0)$ in $W \times \text{conj}(W)$, apply the identity principle for real analytic functions to \eqref{EQ-nopureterms} to get
\begin{equation}\label{EQ-intermediatekernel}
    K_W(z, v) = c_0 \phi_1(z) \overline{\phi_1(v)} \big( K_{\widetilde\Omega}(z, v) \big)^\lambda\ ,\quad z, v \in W.
\end{equation}
By the properties of the Bergman kernel and our choice of complete orthonormal system, we also have for $z, v \in W$
\begin{align*}
    K_W(z, 0) = \phi_1(z) \phi_1(0)\ ,\quad K_W(0, v) = \overline{\phi_1(v)} \phi_1(0)
\end{align*}
so setting $c := c_0 (\phi_1(0))^{-2} \in \mathbb{R}^{> 0}$, \eqref{EQ-intermediatekernel} becomes
\begin{equation}\label{EQ-Bergmanproduct}
    K_W(z, v) = c K_W(z, 0) K_W(0, v) \big( K_{\widetilde\Omega}(z, v) \big)^\lambda\ ,\quad z, v \in W.
\end{equation}
By the reproducing property \eqref{EQ-Reprodkernel} of $K_W$ applied to $1 \in A^2(W)$ and \eqref{EQ-Bergmanproduct}, for all $z \in W$,
\begin{align*}
    1 = c K_W(z, 0) \int_W \big( K_{\widetilde\Omega}(z, v) \big)^\lambda K_W(0, v)\ dV(v).
\end{align*}
Then $K_W(\cdot, 0) \in \mathcal{O}(W)$ is nowhere zero on $W$. By Proposition \ref{PROP-symmetricdomains} (C), the integral $\int_W \big( K_{\widetilde\Omega}(z, v) \big)^\lambda K_W(0, v)\ dV(v)$ is convergent for any $z \in \widetilde\Omega$, and hence it is holomorphic on $\widetilde\Omega$. We can thus extend $\frac{1}{K_W(\cdot, 0)}$ to $h_0 \in \mathcal{O}(\widetilde\Omega)$. Set $Z(h_0) := h_0^{-1}(0) \subset \widetilde\Omega \setminus W$. 
Now we show that $E := \partial W \cap \widetilde{\Omega} \subset \mathbb{C}^n$ is pluripolar, which means by \cite[Josefson's theorem]{Josefson} that for every point $a \in E$ there exist an open neighborhood $U \subset \mathbb{C}^n$ of $a$ and $u \in \mathrm{PSH}(U)$, not identically equal to $-\infty$, such that
\[
E \cap U \subset \{ z \in U \ | \  u(z) = -\infty \}.
\]
Set
\begin{align*}
    L_1 := \partial W \cap Z(h_0)\ ,\quad L_2 := \partial W \cap (\widetilde\Omega \setminus Z(h_0))
\end{align*}
so that $E = L_1 \cup L_2$, and since $L_1$ is pluripolar, it only suffices to prove that $L_2$ is pluripolar. Let $z_0 \in L_2$ and $B_\varepsilon(z_0) \subset \widetilde\Omega \setminus Z(h_0)$. Then $\frac{1}{h_0} \in \mathcal{O}(B_\varepsilon(z_0) \cap W)$, so there exist $\varepsilon' \in (0, \varepsilon)$ and $t \in \mathbb{R}^{> 0}$ such that
\begin{align*}
    \frac{1}{| h_0(z) |^2} \leq t\ \quad \forall\, z \in B_{\varepsilon'}(z_0) \cap W.
\end{align*}
Consequently, by \eqref{EQ-Bergmanproduct}
\begin{align*}
    \limsup_{z \in W,\ z \to z_0} K_W(z, z) &= \inf_{\delta > 0} \left\{ \sup_{z \in B_{\delta}(z_0) \cap W} \left\{ \frac{c}{| h_0(z) |^2} \big(K_{\widetilde\Omega}(z, z) \big)^\lambda \right\} \right\} \leq\\
    &\leq \sup_{z \in B_{\varepsilon'}(z_0) \cap W} \left\{ \frac{c}{| h_0(z) |^2} \big(K_{\widetilde\Omega}(z, z) \big)^\lambda \right\} \leq\\
    &\leq ct \sup_{z \in B_{\varepsilon'}(z_0)} \left\{ \big(K_{\widetilde\Omega}(z, z) \big)^\lambda \right\} < \infty.
\end{align*}
Thus, if $U_{z_0}$ is a neighborhood of $z_0$, by  \cite[Lemma~11]{PflugZwonek}: $P_{z_0} := U_{z_0} \setminus W$ is pluripolar. In particular, we now show
\begin{align*}
    P_{z_0} = \partial W \cap U_{z_0}.
\end{align*}
On one hand, since $W$ is open in $\mathbb{C}^n$: $\partial W \cap W = \emptyset$, hence 
\begin{align*}
    \partial W \cap U_{z_0} \subseteq U_{z_0} \cap (\mathbb{C}^n \setminus W) = P_{z_0}.
\end{align*}
On the other hand, if by contradiction there exist $a \in P_{z_0} \cap (\mathbb{C}^n \setminus \overline{W})$, then there exists $U_a$ neighborhood of $a$ such that
\begin{align*}
    U_a \subset U_{z_0} \cap (\mathbb{C}^n \setminus \overline{W}) \subseteq U_{z_0} \cap (\mathbb{C}^n \setminus W) = P_{z_0}
\end{align*}
which is absurd, because it is well-known (see, e.g., \cite[Theorem~4.17]{Demailly}) that a pluripolar set has zero Lebesgue measure. It is then proved that $L_2$ is pluripolar, as desired. To conclude the proof, we now show
$\widetilde\Omega \setminus E = W$. Clearly $\widetilde\Omega \setminus E \supseteq W$. Conversely, just notice that
\begin{align*}
    \widetilde\Omega \setminus E \subset W \cup (\mathbb{C}^n \setminus \overline{W})
\end{align*}
and if $(\widetilde\Omega \setminus E) \cap (\mathbb{C}^n \setminus \overline{W}) \neq \emptyset$, we would have a separation of $\widetilde\Omega \setminus E$, which contradicts the well-known fact (see, e.g., \cite[Corollary~5.26]{Demailly})  that a pluripolar and relatively closed subset of a domain does not disconnect the domain.

\subsection{Final remarks}

\begin{remark}\label{RMK-Bergmankernelcomplement}\rm
Notice that the converse of Theorem~\ref{THM-completelocsymmetric} clearly holds, since the Bergman metric of a bounded symmetric domain (and, more generally, the metric of a symmetric space) is complete. We claim that an analogous ``converse'' statement for Theorem~\ref{THM-pseudoconvexlocsymmetric} is also true.

Indeed, in the proof of Theorem~\ref{THM-pseudoconvexlocsymmetric}, by \cite[Lemma~1]{Irgens}, the basic properties of the Bergman kernel, and the fact that a pluripolar set has zero Lebesgue measure, the restriction map
\[
A^2(\widetilde\Omega) \longrightarrow A^2(\widetilde\Omega\setminus E),
\qquad 
\psi \longmapsto \psi|_{\widetilde\Omega\setminus E},
\]
is an isomorphism of Hilbert spaces. In particular,
$
K_{\widetilde\Omega\setminus E}
=
K_{\widetilde\Omega}\big|_{(\widetilde\Omega\setminus E)\times(\widetilde\Omega\setminus E)},
$
and hence the associated Bergman metrics satisfy
$g_{\widetilde\Omega\setminus E}
=
g_{\widetilde\Omega}\big|_{\widetilde\Omega\setminus E}.
$

Therefore, if $\widetilde\Omega\subset\mathbb{C}^n$ is a bounded symmetric domain and $E\subset \widetilde\Omega$ is pluripolar and closed in $\widetilde\Omega$, then $(\widetilde\Omega\setminus E, g_{\widetilde\Omega\setminus E})$ inherits the same local symmetry: namely, since $g_{\widetilde\Omega\setminus E}$ coincides with the restriction of $g_{\widetilde\Omega}$, and the curvature tensor of the Bergman metric on a bounded symmetric domain is parallel, we obtain
$\nabla R^{\widetilde\Omega\setminus E} = 0,
$
which proves the claim.
\end{remark}

\begin{remark}\label{RMK-Generalization}\rm
Notice that the authors of \cite{UnifThm2} extend the uniformization theorem mentioned above for bounded pseudoconvex domains to Stein manifolds whose Bergman metric has constant holomorphic sectional curvature, provided that their Bergman space satisfies the following conditions:
\begin{itemize}
    \item\label{item:first} it is nontrivial and base-point free;
    \item\label{item:second} it separates points;
    \item\label{item:third} it separates holomorphic directions.
\end{itemize}
These assumptions guarantee that the Bergman kernel and the Bergman metric are well-defined, and they allow one to define the Bergman--Bochner map, which in this setting is injective.
Using similar techniques, one can extend Theorem~\ref{THM-completelocsymmetric} and Theorem~\ref{THM-pseudoconvexlocsymmetric} to Stein manifolds satisfying the above conditions and the local symmetry of the Bergman metric hypothesis.
\end{remark}

\begin{remark}\label{RMK-OpenProblem}\rm
With regards to \cite{UnifThm3}, it is natural to ask whether a parallel result can be obtained by removing the assumption of pseudoconvexity of the domain. In Theorem \ref{THM-pseudoconvexlocsymmetric} we proved, using only the local symmetry condition, that the bounded locally symmetric domain is biholomorphic to an open subset of a bounded symmetric domain with the property that the Bergman metric of such open subset is, after eventually rescaling, the Bergman metric of the bounded symmetric domain. Since the analogue of this fact in the constant holomorphic sectional curvature setting is the starting point in \cite{UnifThm3}, we pose the following:
\begin{conjecture}
Let $\Omega \subset \mathbb{C}^n$, $n \geq 1$, be a bounded domain. If $\nabla R^\Omega = 0$, then there exist
$\widetilde\Omega \subset \mathbb{C}^n$ bounded symmetric domain and  $E \subset \widetilde\Omega$ of zero Lebesgue measure over which $A^2(\widetilde\Omega \setminus E)$ functions extend holomorphically to $\widetilde\Omega$, such that $\Omega \cong \widetilde\Omega \setminus E$. 
\end{conjecture}
\end{remark}

\end{document}